\documentclass[conference]{IEEEtran}

\usepackage{pifont}
\usepackage{bbding}  

\IEEEoverridecommandlockouts                              

\usepackage{amsmath,amssymb,epsf,epsfig,times}
\usepackage{amsfonts}
\usepackage[all]{xy}
\usepackage{color}
\usepackage{subfigure}
\usepackage{url,cite}

\newtheorem{theorem}{Theorem}[section]
\newtheorem{lemma}{Lemma}[section]

\newcommand{\norm}[1]{\left\|#1\right\|}

\newtheorem{definition}[theorem]{Definition}
\newtheorem{assumption}[theorem]{Assumption}

\newtheorem{remark}[theorem]{Remark}

\newcommand{\OMIT}[1]{}

\newif\ifpdf
\ifx\pdfoutput\undefined \pdffalse \else \pdfoutput=1 \pdftrue \fi
\ifpdf
\usepackage{graphicx}
\usepackage{epstopdf}
\usepackage{graphicx}
\fi
\title{\LARGE \bf Distributed Convex Optimization of Time-Varying Cost Functions with Swarm Tracking Behavior for Continuous-time Dynamics}

\author{Salar Rahili, Wei Ren, Sheida Ghapani
\thanks{Salar Rahili and Wei Ren are with the Department of Electrical Engineering, University of California, Riverside, CA, USA.
       {Email: srahi001@ucr.edu, ren@ee.ucr.edu, sghap001@ucr.edu}}
}
\begin{document}
%

\maketitle

\begin{abstract}
In this paper, a distributed convex optimization problem with swarm tracking behavior is studied for continuous-time multi-agent systems. The agents' task is to drive their center to track an optimal trajectory which minimizes the sum of local time-varying cost functions through local interaction, while maintaining connectivity and avoiding inter-agent collision. Each local cost function is only known to an individual agent and the team's optimal solution is time-varying. Here two cases are considered, single-integrator dynamics and double-integrator dynamics. For each case, a distributed convex optimization algorithm with swarm tracking behavior is proposed where each agent relies only on its own position and the relative positions (and velocities in the double-integrator case) between itself and its neighbors. It is shown that the center of the agents tracks the optimal trajectory, the the connectivity of the agents will be maintained and inter-agent collision is avoided. Finally, numerical examples are included for illustration.
\end{abstract}


\IEEEpeerreviewmaketitle

\section{Introduction} \label{sec:introduction}
Flocking or swarm tracking of a leader has received significant attention in the literature \cite{saber,su,caoren12}. The goal is that a group of agents tracks a leader only with local interaction while maintaining connectivity and avoiding inter-agent collision. A swarm tracking algorithm is studied in \cite{saber}, where it is assumed that the leader has a constant velocity and is a neighbor of all followers. The result in \cite{saber} has been extended in \cite{su} for the leader with a time-varying velocity, where it requires the leader to be a neighbor of all followers too. In \cite{caoren12}, a swarm tracking algorithm via a variable structure approach is introduced, where the leader is a neighbor of only a subset of the followers. In the aforementioned studies, the leader plans the trajectory for the team and no optimal criterion is defined. However, in many multi-agents applications it is relevant for the agents to cooperatively optimize a certain criterion.

In the distributed convex optimization literature, there exists a significant interest in a class of problems, where the goal is to minimize the sum of local cost functions, each of which is known to only an individual agent. Recently some remarkable results based on the combination of consensus and subgradient algorithms have been published \cite{D3-08, D4-10,D8-11}. For example, this combination is used in \cite{D3-08} for solving the coupled optimization problem with a fixed undirected topology. A projected subgradient algorithm is proposed in \cite{D4-10}, where each agent is required to lie in its own convex set. It is shown that all agents can reach an optimal point in the intersection of all agents' convex sets for a time-varying communication graph with doubly stochastic edge weight matrices.

However, all the aforementioned works are based on discrete-time algorithms. Recently, some new research is conducted on distributed optimization problems for multi-agent systems with continuous-time dynamics. Such a scheme has applications in motion coordination of multi-agent systems. For example, multiple physical vehicles modelled by continuous-time dynamics might need to rendezvous at a team optimal location. In \cite{C3-11}, a generalized class of zero-gradient sum controllers for twice differentiable strongly convex functions with an undirected graph is introduced. In \cite{C4-13}, a continuous version of \cite{D4-10} is studied, where it is assumed that each agent is aware of the convex solution set of its own cost function and the intersection of all these sets is nonempty. In \cite{C7-12}, the convergence rate and error bounds of a continuous-time distributed optimization algorithm has been derived. In \cite{C6-11}, an approach is given to address the problem of distributed convex optimization with equality and inequality constraints. A proportional-integral algorithm is introduced in \cite{C5, C2-12, C2-14}, where \cite{C2-12} considers strongly connected weight balanced directed graphs and \cite{C2-14} extends these results using a discrete-time communications scheme. A distributed optimization problem for single-integrator agents is studied in \cite{C1} with the adaptivity and finite-time convergence properties.

Having time-invariant cost functions is a common assumption in the literature. However, in many applications the local cost functions are time varying, reflecting the fact that the optimal point could be changing over time and creates a trajectory. In addition, in continuous-time optimization problems, the agents are usually assumed to have single-integrator dynamics. However, a broad class of vehicles requires double-integrator dynamic models. In our early work \cite{acc}, a preliminary attempt for time-varying cost function is made. However, in all articles on distributed optimization mentioned above, the agents will eventually approach a common optimal point. While the algorithms can be applied to rendezvous problems, they are not applicable to more complicated swarm tracking problems.

In this paper, we study a distributed convex optimization problem with swarm tracking behavior for continuous-time multi-agent systems. There exist significant challenges in the study due to the coexistence of nonlinear swarm behavior, optimization objectives, and time-varying cost functions under the constraints of local information and local interaction.  The center of the agents will track an optimal trajectory which minimizes the sum of local time-varying cost functions through local interaction. The agents will maintain connectivity and avoid inter-agent collision. Each local cost function is only known to an individual agent and the team's optimal solution is time-varying. Both cases of single-integrator
dynamics and double-integrator dynamics will be considered.

The remainder of this paper is organized as follows: In Section \ref{sec:notation}, the notation and preliminaries used throughout this paper are introduced. In Section \ref{sec:single} and \ref{sec:double}, two distributed convex optimization algorithms with swarm tracking behavior and time-varying cost functions for respectively single-integrator and double-integrator dynamics are designed. Finally in Section \ref{sec:sim}, numerical examples are given for illustration.

\section{notations and preliminaries} \label{sec:notation}
The following notations are adopted throughout this paper. $R^+$ denotes the positive real numbers. ${\mathcal I}$ denotes the index set $\{1,...,N\}$; The transpose of matrix $A$ and vector $x$ are shown as $A^T$ and $x^T$, respectively. $||x||_p$ denotes the p-norm of the vector $x$. Let $\textbf{1}_n$ and $\textbf{0}_n$ denote the column vectors of $n$ ones and zeros, respectively and $I_n$ denote the $n \times n$ identity matrix. For matrix $A$ and $B$, the Kronecker product is denoted by $A \otimes B$. The gradient and Hessian of function $f$ are denoted by $\nabla f$ and $H$, respectively. The matrix inequality $A>(\geq) B$ means that $A-B$ is positive (semi-)definite.

Let a triplet ${\mathcal G}=({\mathcal V},{\mathcal E},{\mathcal A})$ be an undirected graph, where ${\mathcal V}=\{1,...,N\}$ is the node set and ${\mathcal E} \subseteq {\mathcal V} \times {\mathcal V}$ is the edge set, and ${\mathcal A}=[a_{ij}] \in R^{N \times N}$ is a weighted adjacency matrix. An edge between agents $i$ and $j$, denoted by $(i,j) \in {\mathcal E}$, means that they can obtain information from each other. The weighted adjacency matrix ${\mathcal A}$ is defined as $a_{ii}=0$, and $a_{ij}=a_{ji}>0$ if $(i,j) \in {\mathcal E}$ and $a_{ij}=0$, otherwise. The set of neighbors of agent $i$ is denoted by $N_i=\{j \in {\mathcal V}: (j,i) \in {\mathcal E} \}$. A sequence of edges of the form $(i,j),(j,k),...,$ where $i,j,k \in {\mathcal V},$ is called a path. The graph ${\mathcal G}$ is connected, if there is a path from every node to every other node. The incidence matrix associated with the graph ${\mathcal G}$ is represented as $D$. Let the Laplacian matrix $L=[l_{ij}] \in R^{N \times N}$ associated with the graph ${\mathcal G}$ be defined as $l_{ii} =\sum_{j=1,j \neq i}^N a_{ij}$ and $l_{ij}=-a_{ij}$ for $i \neq j$. The Laplacian matrix $L$ is symmetric positive semidefinite. We denote the eigenvalues of $L$ by $\lambda_1, ...,\lambda_N$. The undirected graph ${\mathcal G}$ is connected if and only if $L$ has a simple zero eigenvalue with the corresponding eigenvector $\textbf{1}_N$ and all other eigenvalues are positive \cite{graphtheory}. When the graph ${\mathcal G}$ is connected, we order the eigenvalues of $L$ as $\lambda_1=0 < \lambda_2 \leq ...\leq \lambda_N$. Note that $L = D D^T$.
\begin{lemma}\cite{boydopt} \label{lem-intro-gradzero}
Let $f(x): R^{m} \rightarrow R$ be a continuously differentiable convex function. $f(x)$ is minimized if and only if $\nabla f=0$.
\end{lemma}
\begin{definition} \cite{boydopt}
$f(x)$ is m-strongly convex if and only if 
\begin{equation*}
(y-x)(\nabla f(y)-\nabla f(x)) \geq m \norm{y-x}^2,
\end{equation*} 
for $ m>0,\forall x,y \in R^{n}, x\neq y$. If $f(x)$ is m-strongly convex and twice differentiable on $x$, then $H(x) \geq mI_n$.
\end{definition}
\begin{lemma} \cite{eig1} \label{eigvalue}
The second smallest eigenvalue $\lambda_2$ of the Laplacian matrix $L$ associated with the undirected connected graph ${\mathcal G}$ satisfies $\lambda_2 = \min_{x^T \textbf{1}_N =0, x \neq \textbf{0}_N} \frac{x^T L x}{x^T x}$.
\end{lemma}
\begin{lemma} \cite{boyd} \label{schur}
The symmetric matrix
\begin{equation} \nonumber
\left( {\begin{array}{cc} Q & S \\ S^T & R \\ \end{array} } \right)
\end{equation}
is positive definite if and only if one of the following conditions holds:
(i) $Q> 0, R - S^T Q^{-1} S > 0$; or (ii) $R >0, Q - S R^{-1} S^T > 0.$
\end{lemma}
\begin{assumption} \label{as1}
The function $f_0(x,t)$ is m-strongly convex and continuously twice differentiable with respect to $x,\ \forall x,t$.
\end{assumption}

\section{Time-Varying Convex Optimization with Swarm Tracking For Single-Integrator Dynamics} \label{sec:single}

Consider a multi-agent system consisting of $N$ physical agents with an interaction topology described by the undirected graph ${\mathcal G}$. Suppose that the agents satisfy the continuous-time single-integrator dynamics
\begin{equation} \label{single}
 \dot{x}_i(t) =u_i(t)
\end{equation}
where $x_i(t) \in R^{m}$ is the position of agent $i$, and $u_i(t) \in R^{m}$ is the control input of agent $i$. As $x_i(t)$ and $u_i(t)$ are functions of time, we will write them as $x_i$ and $u_i$ for ease of notation. A time-varying local cost function $f_i : R^{m}\times R^{+} \rightarrow R$ is assigned to agent $i\in {\mathcal I},$ which is known to only agent $i$. The team cost function $f : R^{m}\times R^{+} \rightarrow R$ is denoted by \begin{equation} \label{teamcost}
 f(x,t)\triangleq \sum_{i=1}^N f_i(x,t).
\end{equation}Our objective is to design $u_i$ for \eqref{single} using only local information and local interaction with neighbors such that the center of all agents tracks the optimal state $x^*(t)$, and the agents maintain connectivity while avoiding inter-agent collision. Here $x^*(t)$ is the minimizer of the time-varying convex optimization problem
\begin{equation} \label{xstar1}
x^*(t)= \text{arg} \min_{x \in R^m} f(x,t).
\end{equation} 
The problem defined in \eqref{xstar1} is equivalent to
\begin{equation} \label{costdis}
\min_{x_i(t)} \sum_{i=1}^N f_i(x_i,t) \ \text{subject to} \ x_i=x_j.
\end{equation} 
Intuitively, the problem is deformed as a consensus problem and a minimization problem on the team cost function \eqref{teamcost}.

%
In our proposed algorithm, each agent has access to only its own position and the relative positions between itself and its neighbors. To solve this problem, we propose the algorithm 
\begin{equation} \label{usingleswarm}
\begin{split}
u_i(t) =& -\alpha \sum_{j \in N_i} \frac{\partial V_{ij}}{\partial x_i}- \beta \text{sgn}( \sum_{j \in N_i} \frac{\partial V_{ij}}{\partial x_i})+\phi_i, 
\end{split}
\end{equation}
where
\begin{equation*}
\phi_i \triangleq -  H^{-1}_i(x_i,t)\big(\tau \nabla f_i(x_i,t)+ \frac{\partial \nabla f_i(x_i,t)}{\partial t}\big),
\end{equation*}
$V_{ij}$ is a potential function between agents $i$ and $j$ to be designed, $\alpha$ is non-negative, $\beta$ is positive, and sgn($\cdot$) is the signum function defined componentwise. It is worth mentioning that $\phi_i$ depends on only agent $i$'s position. We assume that each agent has a radius of communication/sensing $R$, where if $\norm{x_i-x_j}<R$ agent $i$ and $j$ become neighbors. Our proposed algorithm guarantees connectivity maintenance which means that if the graph ${\mathcal G}(0)$ is connected, then for all $t$, ${\mathcal G}(t)$ will remain connected. Before our main result, we need to define the potential function $V_{ij}$.

\begin{definition} \label{defswarm} \cite{caoren12}
The potential function $V_{ij}$ is a differentiable nonnegative function of $\norm{x_i-x_j}$ which satisfy the following conditions
\begin{itemize}
\item[1)] $V_{ij}=V_{ji}$ has a unique minimum in $\norm{x_i-x_j}=d_{ij}$, where $d_{ij}$ is a desired distance between agents $i$ and $j$ and $R> \max_{i,j} d_{ij}$,
\item[2)] $V_{ij} \rightarrow \infty$ if $\norm{x_i-x_j} \rightarrow 0$.
\item [3)] $V_{ii}=c$, where $c$ is a constant. \vspace*{0.1cm}
\item [4)] \small $\begin{cases} \frac{\partial V_{ij}}{\partial(\norm{x_i-x_j})} =0 &  \norm{x_i(0)-x_j(0)}\geq R,\norm{x_i-x_j}\geq R,
\\ \frac{\partial V_{ij}}{\partial(\norm{x_i-x_j})} \to \infty & \norm{x_i(0)-x_j(0)}< R, \norm{x_i-x_j} \to R, 
\end{cases}$ \normalsize
\end{itemize}
\end{definition}\vspace*{0.1cm}

\begin{theorem} \label{theoremswarm}
Suppose that graph ${\mathcal G}(0)$ is connected, Assumption \ref{as1} holds for each agent's cost function $f_i(x_i(t),t),$ and the gradient of the cost functions can be written as $\nabla f_i(x_i,t)=\sigma x_i+g_i(t), \ \forall i \in {\mathcal I}$. If $\alpha \geq 0$, and $\beta \geq \norm{\phi_i}_1, \ \forall i \in {\mathcal I}$, for system \eqref{single} with algorithm \eqref{usingleswarm}, the center of the agents tracks the optimal trajectory while maintaining connectivity and avoiding inter-agent collision. 
\end{theorem}
\begin{proof}
Define the positive semi-definite Lyapanov function candidate
 \begin{equation*}
W= \frac{1}{2}\sum_{i=1}^N \sum_{j=1}^N V_{ij}.
\end{equation*}
The time derivative of $W$ is obtained as
\begin{equation*} \small 
\begin{split}
\dot{W}= & \frac{1}{2}\sum_{i=1}^N \sum_{j=1}^N \big(\frac{\partial V_{ij}}{\partial x_i} \dot{x}_i + \frac{\partial V_{ij}}{\partial x_j} \dot{x}_j \big)= \sum_{i=1}^N \sum_{j=1}^N \frac{\partial V_{ij}}{\partial x_i} \dot{x}_i
\end{split}
\end{equation*}\normalsize
where in the second equality, Lemma 3.1 in \cite{caoren12} has been used. Now, rewriting $\dot{W}$ along with the close-loop system \eqref{usingleswarm} and \eqref{single} we have 
\begin{equation*} \small 
\begin{split}
&\dot{W}=  -\alpha \sum_{i=1}^N \big( \sum_{j=1}^N \frac{\partial V_{ij}}{\partial x_i} \big)^2 - \beta \sum_{i=1}^N \sum_{j=1}^N \frac{\partial V_{ij}}{\partial x_i} \text{sgn}( \sum_{j=1}^N \frac{\partial V_{ij}}{\partial x_i}) \\
&+\sum_{i=1}^N  \sum_{j=1}^N \frac{\partial V_{ij}}{\partial x_i} \phi_i \leq  \sum_{i=1}^N \bigg(\norm{ \sum_{j=1}^N \frac{\partial V_{ij}}{\partial x_i}}_1 (\norm{\phi_i}_1 -\beta) \bigg)
\end{split}
\end{equation*}\normalsize
It is easy to see that if $\beta \geq \norm{\phi_i}_1 \ \forall i \in {\mathcal I},$ then $\dot{W}$ is negative. Therefore, having $W\geq 0$ and $\dot{W} \leq 0$, we can conclude that $V_{ij} \in \mathcal{L}_{\infty}$. Since $V_{ij}$ is bounded, based on Definition \ref{defswarm}, it is guaranteed that there will not be a inter-agent collision and the connectivity is maintained.

In what follows, we focus on finding the relation between the optimal trajectory and the agents' positions. 
Define the Lyapanov candidate function
\begin{equation*}  W_1= \frac{1}{2} ( \sum_{j=1}^N \nabla f_j(x_j,t))^T (\sum_{j=1}^N \nabla f_j(x_j,t))\end{equation*} \normalsize
where $W_1$ is positive semi-definite. The time derivative of $W_1$ can be obtained as
\small\begin{equation} 
\dot{W}_1= (\sum_{j=1}^N \nabla f_j(x_j,t))^T (\sum_{j=1}^N H_j(x_j,t) \dot{x}_j+ \sum_{j=1}^N \frac{\partial }{\partial t} \nabla f_j(x_j,t)) \nonumber
\end{equation} \normalsize
Because $\nabla f_i(x_i,t)=\sigma x_i+g_i(t), \ \forall i \in {\mathcal I}$, we know $H_i(x_i,t)=H_j(x_j,t)$ and we obtain
\small\begin{equation}  \label{local2}
\begin{split}
\dot{W_1}= (\sum_{j=1}^N \nabla &f_j(x_j,t))^T (H_i(x_i,t))\\&\big(\sum_{j=1}^N  \dot{x}_j+ H_i^{-1}(x_i,t)\sum_{j=1}^N \frac{\partial }{\partial t} \nabla f_j(x_j,t)\big). \end{split} \end{equation} \normalsize
Based on Definition \ref{defswarm}, we can obtain
\begin{equation} \label{rondVijrelation} 
\frac{\partial V_{ij}}{\partial e_{X_i}}=\frac{\partial V_{ji}}{\partial e_{X_i}}=-\frac{\partial V_{ij}}{\partial e_{X_j}}
\end{equation}\normalsize
Now, by summing both sides of the closed-loop system \eqref{single} with control algorithm \eqref{usingleswarm}, for $i=1,...,N$, we have $\sum_{j=1}^N \dot{x}_j = \sum_{j=1}^N \phi_j$. Hence we can rewrite \eqref{local2} as
\begin{equation} 
\dot{W}_1= -\tau(\sum_{j=1}^N \nabla f_j(x_j,t))^T (\sum_{j=1}^N \nabla f_j(x_j,t)) \nonumber \end{equation} \normalsize
Therefore, $\dot{W}_1<0$ for $\sum_{j=1}^N \nabla f_j(x_j) \neq 0$. This guarantees that $\sum_{j=1}^N \nabla f_j(x_j)$ will asymptomatically converge to zero. Because $\nabla f_i(x_i,t)=\sigma x_i+g_i(t)$, we have $ \sum_{j=1}^N x_i=\frac{-\sum_{j=1}^N g_j}{ \sigma}$. On the other hand, using Lemma \ref{lem-intro-gradzero}, we know $\sum_{j=1}^N \nabla f_j(x^*,t)=0$. Hence, the optimal trajectory is
\begin{equation} \label{xvoptimal}  \small 
 x^*=\frac{-\sum_{j=1}^N g_j}{N \sigma}, 
\end{equation} \normalsize  which implies that 
\begin{equation} \label{xvbound}  \small
\begin{split} 
& x^*=\frac{1}{N}\sum_{j=1}^N x_i.\\
\end{split} \end{equation} \normalsize
where we have shown that the center of the agents will track the team cost function minimizer.
\end{proof}

\begin{remark}
Connectivity maintenance guarantees that there exists a path between each two agents $i, j \in {\mathcal I}\ \forall t$. Hence, we have $\norm{x_i - \frac{1}{N} \sum_{j=1}^N x_j}  <(N-1)R$. Now, using the result in Theorem \ref{theoremswarm}, it is easy to see that \begin{equation*} \small
\norm{x^*(t) - x_i(t)}  <NR ,\forall i \in {\mathcal I}.
\end{equation*} \normalsize which guarantees a bounded error between each agent's position and the optimal trajectory.
\end{remark}
\begin{remark} \label{Remarkboundsingle}
There exists a sufficient condition on the agents' cost functions to make sure that $\norm{\phi_i}_1, \ \forall i \in {\mathcal I}$ is bounded. If both $\norm{\nabla f_i}$ and $\norm{\frac{\partial \nabla  f_i}{\partial t}} \ \forall i,j \in {\mathcal I}$ are bounded, then $\norm{\phi_i}_1 \ \forall i \in {\mathcal I}$ is bounded.
\end{remark}
\begin{remark} \label{Remarksingle2}
In Theorem \ref{theoremswarm}, it is required that each agent's cost function have a gradient in the form of $\nabla f_i(x_i,t)=\sigma x_i+g_i(t)$. While this can be a restrictive assumption, there exists an important class of cost functions that satisfy this assumption. For example, the cost functions that are commonly used for energy minimization, e.g., $f_i(x_i,t)=(ax_i+g_i(t))^{2n},$ where $a \in R, n \in \{1,2,...\}$ are constants and $g_i(t)$ is a time-varying function particularly for agent $i$. Here, to satisfy the condition $\beta >\norm{\phi_i}_1 \ \forall i \in {\mathcal I}$, as discussed in Remark \ref{Remarkboundsingle} it is sufficient to have a bound on $\norm{g_i(t)}$ and $\norm{\dot{g}_i(t)}$. 
\end{remark}
\section{Time-Varying Convex Optimization with Swarm Tracking Behavior For Double-Integrator Dynamics}  \label{sec:double}
In this section, we study distributed convex optimization of time-varying cost functions with swarm behavior for double-integrator dynamics. Suppose that the agents satisfy the continuous-time double-integrator dynamics
\begin{equation} \label{double}
 \begin{cases} \dot{x}_i(t) = v_i(t) \\ 
\dot{v}_i(t) = u_i(t) \end{cases}
\end{equation}
where $x_i, v_i \in R^{m}$ are, respectively, the position and velocity of agent $i$, and $u_i \in R^{m}$ is the control input of agent $i$. 

%
We will propose an algorithm, where each agent has access to only its own position and the relative positions and velocities between itself and its neighbors. We propose the algorithm 
\begin{equation} \label{uswarm}
\begin{split}
u_i(t) =& - \sum_{j \in N_i} \frac{\partial V_{ij}}{\partial x_i}- \alpha\sum_{j \in N_i}(v_i -v_j)\\& -\beta  \sum_{j \in N_i} \text{sgn}(v_i -v_j)+\phi_i, 
\end{split}
\end{equation}
where 
\begin{equation*} 
\begin{split} \small
\phi_i=-H_i^{-1} (x_i,t)\big( \frac{\partial}{\partial t}\frac{d \nabla f_i(x_i,t)}{dt}+\frac{d \nabla f_i(x_i,t)}{dt} \big) \ \ \ &\\
- H_i \nabla f_i(x_i,t)+ \bigg(H_i^{-1}(x_i,t) (\frac{d}{dt} H_i(x_i,t)) H^{-1}_i(x_i,t)&\bigg) \\
\big( \frac{\partial \nabla f_i (x_i,t)}{\partial t}+\nabla f_i(& x_i,t) \big). 
\end{split}
\end{equation*}
The potential function $V_{ij}$ is introduced in Definition \ref{defswarm}, $\alpha$ and $\beta$ are positive coefficients.

\begin{theorem} \label{theoremswarmdouble}
Suppose that graph ${\mathcal G}(0)$ is connected, Assumption \ref{as1} holds for each agent's cost function $f_i(x_i(t),t),$ and the gradient of the cost functions can be written as $\nabla f_i(x_i,t)=\sigma x_i+g_i(t), \ \forall i \in {\mathcal I}$. If $\beta \geq \frac{ \norm{(\Pi \otimes I_m)\Phi}_2}{\sqrt{\lambda_{2}(L)}}$, for system (\ref{double}) with algorithm \eqref{uswarm}, the center of the agents tracks the optimal trajectory and the agents' velocities track the optimal velocity while maintaining connectivity and avoiding inter-agent collision.  
\end{theorem} \vspace*{0.1cm}

\begin{proof}
The closed-loop system \eqref{double} with control input \eqref{uswarm} can be recast into a compact form as
\begin{equation} \label{closeloopswarm}
\begin{cases} \dot{X} =& V \\ 
\dot{V} =&-\alpha(L\otimes I_m) V-\beta D \text{sgn}(D^T V)
\\& \begin{pmatrix}
  \sum_{j \in N_1} \frac{\partial V_{1j}}{\partial x_1} \\
    \vdots \\
   \sum_{j \in N_N} \frac{\partial V_{Nj}}{\partial x_N}
 \end{pmatrix} +(\Pi\otimes I_m) \Phi
\end{cases} 
\end{equation}
where $X=[x_1^T,x_2^T,...,x_N^T]^T,$ and $V=[v_1^T,v_2^T,...,v_N^T]^T$ are, respectively, positions and velocities of $N$ agents and $\Phi=[\phi_1^T,\phi_2^T,...,\phi_N^T]^T$. It is preferred to rewrite Eq. \eqref{closeloopswarm} in terms of the consensus error. Therefore, we define $e_X(t)=(\Pi \otimes I_m)X$ and $e_V(t)=(\Pi \otimes I_m)V$, where $\Pi=I_N- \frac{1}{N} \mathbf{1}_N \mathbf{1}_N^T$. Note that $\Pi$ has one simple zero eigenvalue with $\mathbf{1}_N$ as its right eigenvector and has $1$ as its other eigenvalue with the multiplicity $N-1$. Then it is easy to see that $e_V(t)=0$ if and only if $v_i=v_j \ \forall i,j \in {\mathcal I}$. Thus the agents' velocities reach consensus if and only if $e_V(t)$ converge to zero asymptotically. Rewriting \eqref{closeloopswarm} we have
\begin{equation} \label{errorsysswarm}
\begin{cases} \dot{e}_X =& e_V \\ 
\dot{e}_V =&-\alpha(L\otimes I_m) e_V-\beta D \text{sgn}(D^T e_V)
\\& \begin{pmatrix}
  \sum_{j \in N_1} \frac{\partial V_{1j}}{\partial e_{X_1}} \\
    \vdots \\
   \sum_{j \in N_N} \frac{\partial V_{Nj}}{\partial e_{X_N}}
 \end{pmatrix} +(\Pi\otimes I_m) \Phi,
\end{cases} 
\end{equation}
where we recall the elements of $e_X$ and $e_V$ as $e_X=[e_{X_1}^T, e_{X_2}^T,..., e_{X_N}^T]^T$ and $e_V=[e_{V_1}^T, e_{V_2}^T,..., e_{V_N}^T]^T$. Define the positive definite Lyapanov function candidate
 \begin{equation*}
W= \frac{1}{N}\sum_{i=1}^N \sum_{j=1}^N V_{ij} +\frac{1}{2} e_V^T e_V.
\end{equation*}
The time derivative of $W$ along \eqref{errorsysswarm} can be obtained as
\begin{equation} \label{dwswarmdbl}
\begin{split}
\dot{W}= & \frac{1}{2}\sum_{i=1}^N \sum_{j=1}^N \big(\frac{\partial V_{ij}}{\partial e_{X_i}}^T e_{V_i} + \frac{\partial V_{ij}}{\partial e_{X_j}}^T e_{V_j} \big)+ e_V^T \dot{e}_V\\
 =& \sum_{i=1}^N \sum_{j=1}^N \frac{\partial V_{ij}}{\partial e_{X_i}}^T e_{V_i}+ e_V^T \dot{e}_V \\
 =& -\alpha e_V^T (L\otimes I_m) e_V+ e_V^T (\Pi \otimes I_m)\Phi \\
 &-\beta e_V^T (D \otimes I_m)\text{sgn}\big((D^T \otimes I_m)e_V \big),
\end{split}
\end{equation}\normalsize
where in the first equality, Lemma 3.1 in \cite{caoren12} has been used. We also have
\begin{equation} \label{sgnnorm}
\begin{split} 
&e_V^T (\Pi \otimes I_m)\Phi-\beta e_V^T (D \otimes I_m)\text{sgn}\big((D^T \otimes I_m)e_V \big)\\
&\leq \norm{e_V}_2 \norm{(\Pi \otimes I_m)\Phi}_2- \beta \norm{(D^T \otimes I_m)e_V}_1\\
&\leq \norm{e_V}_2 \norm{(\Pi \otimes I_m)\Phi}_2- \beta \sqrt{e_V^T (DD^T \otimes I_m)e_V}\\
&\leq \norm{e_V}_2 \norm{(\Pi \otimes I_m)\Phi}_2- \beta \sqrt{\lambda_{2}(L)} \norm{e_V}_2,
\end{split}
\end{equation}
where in the last inequality Lemma \ref{eigvalue} has been used. Now, it is easy to see that if $\beta \sqrt{\lambda_{2}(L)} \geq \norm{(\Pi \otimes I_m)\Phi}_2,$ then $\dot{W}$ is negative semi-definite. Therefore, having $W\geq 0$ and $\dot{W} \leq 0$, we can conclude that $V_{ij}, e_v \in \mathcal{L}_{\infty}$. By integrating both sides of \eqref{dwswarmdbl}, we can see that  $e_v \in \mathcal{L}_{2}$. Now, applying Barbalat's Lemma\cite{sheida22j}, we obtain that $e_V$ asymptotically converges to zero, which means that the agents' velocities reach consensus as $t\rightarrow \infty$. On the other hand, since $V_{ij}$ is bounded, based on Definition \ref{defswarm}, it is guaranteed that there will not be inter-agent collision and the connectivity is maintained. 

In what follows, we focus on finding the relation between the optimal trajectory of the team cost function and the agents' states. Define the Lyapanov candidate function
\begin{equation}
\begin{split}
W_1&=\frac{1}{2}(\sum_{j=1}^N \nabla f_j(x_j,t))^T (\sum_{j=1}^N \nabla f_j(x_j,t))\\&+\frac{1}{2} (\sum_{j=1}^N v_i -\sum_{j=1}^N S_i )^T (\sum_{j=1}^N v_i -\sum_{j=1}^N S_i ),
\end{split}
\end{equation}
where $ S_i={H_i}^{-1}(x_i,t) \big( \frac{\partial }{\partial t} \nabla f_i(x_i,t) + \nabla f_i(x_i,t)\big)$.
The time derivative of $W_1$ along the system defined in \eqref{double} and \eqref{uswarm} can be obtained as
\begin{equation*} \small
\begin{split}
\dot{W}_1=& (\sum_{j=1}^N \nabla f_j(x_j,t))^T (\sum_{j=1}^N H_j(x_j,t) v_j+ \sum_{j=1}^N \frac{\partial }{\partial t} \nabla f_j(x_j,t))\\&+ (\sum_{j=1}^N v_j -\sum_{j=1}^N S_j )^T (\sum_{j=1}^N \dot{v}_j -\sum_{j=1}^N \dot{S}_j )\\
=& (\sum_{j=1}^N \nabla f_j(x_j,t))^T (\sum_{j=1}^N H_j(x_j,t) v_j+ \sum_{j=1}^N \frac{\partial }{\partial t} \nabla f_j(x_j,t))\\&+ (\sum_{j=1}^N v_j -\sum_{j=1}^N S_j )^T (\sum_{j=1}^N \phi_j -\sum_{j=1}^N \dot{S}_j ) \\
=& (\sum_{j=1}^N \nabla f_j(x_j,t))^T (\sum_{j=1}^N H_j(x_j,t) v_j+ \sum_{j=1}^N \frac{\partial }{\partial t} \nabla f_j(x_j,t))\\&-(\sum_{j=1}^N v_j -\sum_{j=1}^N S_j )^T (\sum_{j=1}^N H_j(x_j,t)\nabla f_j(x_j,t)),\nonumber
\end{split}
\end{equation*} \normalsize
where in the second equality, we have used the fact that by summing both sides of the closed-loop system \eqref{double} with controller \eqref{uswarm} for $j=1,2,...,N$, and using \eqref{rondVijrelation}, we obtain that $\sum_{j=1}^N \dot{v}_j = \sum_{j=1}^N \phi_j$. Because $\nabla f_i(x_i,t)=\sigma x_i+g_i(t), \ \forall i \in {\mathcal I}$, we know $H_i(x_i,t)=H_j(x_j,t)$. Hence, we have
\begin{equation} \small
\begin{split}
\dot{W}_1=& - (\sum_{j=1}^N v_j -\sum_{j=1}^N S_j )^T (H_i(x_i,t)\sum_{j=1}^N \nabla f_j(x_j,t))\\
+(\sum_{j=1}^N& \nabla f_j(x_j,t))^T \bigg(H_i(x_i,t)\sum_{j=1}^N  v_j+ \sum_{j=1}^N \frac{\partial }{\partial t} \nabla f_j(x_j,t)\bigg)\\
=&-(\sum_{j=1}^N \nabla f_j(x_j,t))^T (\sum_{j=1}^N \nabla f_j(x_j,t)).
\end{split}
\end{equation} \normalsize
Therefore, $\dot{W}_1<0$ for $\sum_{j=1}^N \nabla f_j(x_j,t) \neq 0$. Having $W_1\geq 0$ and $\dot{W}_1 \leq 0$, we can conclude that $\sum_{j=1}^N \nabla f_j(x_j,t), \big(\sum_{j=1}^N v_i -\sum_{j=1}^N S_i\big) \in \mathcal{L}_{\infty}$. By integrating both sides of $\dot{W}_1=-(\sum_{j=1}^N \nabla f_j(x_j,t))^T (\sum_{j=1}^N \nabla f_j(x_j,t))$, we can see that  $\sum_{j=1}^N \nabla f_j(x_j,t) \in \mathcal{L}_{2}$. Now, applying Barbalat's Lemma, we obtain that $\sum_{j=1}^N \nabla f_j(x_j,t)$ will asymptomatically converge to zero. Now, under the assumption that $\nabla f_i(x_i,t)=\sigma x_i+g_i(t)$, we have $ \sum_{j=1}^N x_i=\frac{-\sum_{j=1}^N g_j}{ \sigma}$ and $\sum_{j=1}^N v_i=\frac{-\sum_{j=1}^N \dot{g}_j}{\sigma}$. 

On the other hand, using Lemma \ref{lem-intro-gradzero}, we know $\sum_{j=1}^N \nabla f_j(x^*,t)=0$. Hence, the optimal trajectory is
\begin{equation} \label{xvoptimal}  \small
\begin{split} 
& x^*=\frac{-\sum_{j=1}^N g_j}{N \sigma}, \ \ v^*=\frac{-\sum_{j=1}^N \dot{g}_j}{N \sigma}\\
\end{split} \end{equation} \normalsize
Now using \eqref{xvoptimal}, we can conclude that
\begin{equation} \label{xvbound}  \small
\begin{split} 
& x^*=\frac{1}{N}\sum_{j=1}^N x_i, \ \ v^*=\frac{1}{N}\sum_{j=1}^N v_i\\
\end{split} \end{equation} \normalsize
Particularly, we have shown that the average of agents' state, positions and velocities, tracks the optimal trajectory. We also have shown that the agents' velocities reach consensus as $t \to \infty$. Thus we have $v_i$ approaches $v^*$ as $t \to \infty$. This completes the proof.
\end{proof}
\begin{remark}
The assumption $\beta  \geq \frac{\norm{(\Pi \otimes I_m)\Phi}_2}{\sqrt{\lambda_{2}(L)}},$ in Theorem \ref{theoremswarmdouble}, can be interpreted as a bound on the difference between the agents' internal signals. This condition is weaker than putting an upper bound on $\phi_i$. If $\norm{\nabla f_j - \nabla f_i}$, $\norm{\frac{d \nabla  f_j}{dt}-\frac{d\nabla  f_i}{dt}} $ and $\norm{\frac{\partial^2 \nabla  f_j}{\partial t^2}-\frac{\partial^2 \nabla  f_i}{\partial t^2}} \ \forall i,j \in {\mathcal I}$ are bounded, then $\norm{(\Pi \otimes I_m) \Phi}_2$ is bounded. For example, the cost function $f_i(x_i(t),t)=(ax_i(t)+g_i(t))^{2n}$ introduced in Remark \ref{Remarksingle2} will satisfy these conditions if $\norm{g_i(t) -g_j(t)}, \norm{\dot{g}_i(t) -\dot{g}_j(t)}$ and $\norm{\ddot{g}_i(t) -\ddot{g}_j(t)}$ are bounded.
\end{remark}
 

\section{Simulation and Discussion } \label{sec:sim}
In this section, we present two simulations to illustrate the theoretical results in previous sections. Consider a team of six agents. We assumed that $R=5$, which means that two agents are neighbors if their this distance is less than $R$. The agents' goal is to have their center minimize the team cost function $\sum_{i=1}^6 f_i (x_i(t),t)$ where $x_i(t)=({r_{x_i}} (t),{r_{y_i}}(t))^T$ is the coordinate of agent $i$ in $2D$ plane. 

In our first example, we apply algorithm \eqref{usingleswarm} for single-integrator dynamics \eqref{single}. The local cost function for agent $i$ is chosen as 
\begin{equation} \label{costfunc1}
f_i(x_i(t),t)= (r_{x_i}(t) - i \text{sin}(0.2t))^2+ (r_{y_i}(t) - i \text{cos}(0.2t))^2,
\end{equation}
For local cost functions \eqref{costfunc1}, Assumption \ref{as1} and the conditions for agents' cost function in Remark \ref{Remarkboundsingle} hold and the gradient of the cost functions can be rewritten as $\nabla f_i(x_i,t)=\sigma x_i+g_i(t)$. To guarantee the collision avoidance and connectivity maintenance, the potential function partial derivatives is chosen as Eqs. (36) and (37) in \cite{caoren12}, where $d_{ij}=0.5 \ \forall i,j$.  Choosing the coefficients in algorithm \eqref{usingleswarm} as $\alpha=2, \beta=5$, and $\tau=1$, and the results are shown in Fig. \ref{sgldisfig}. 

\begin{figure}[t]
\begin{center} \hspace*{-0.6cm}
\vspace{0.1cm} {\scalebox{0.2700}{\includegraphics*{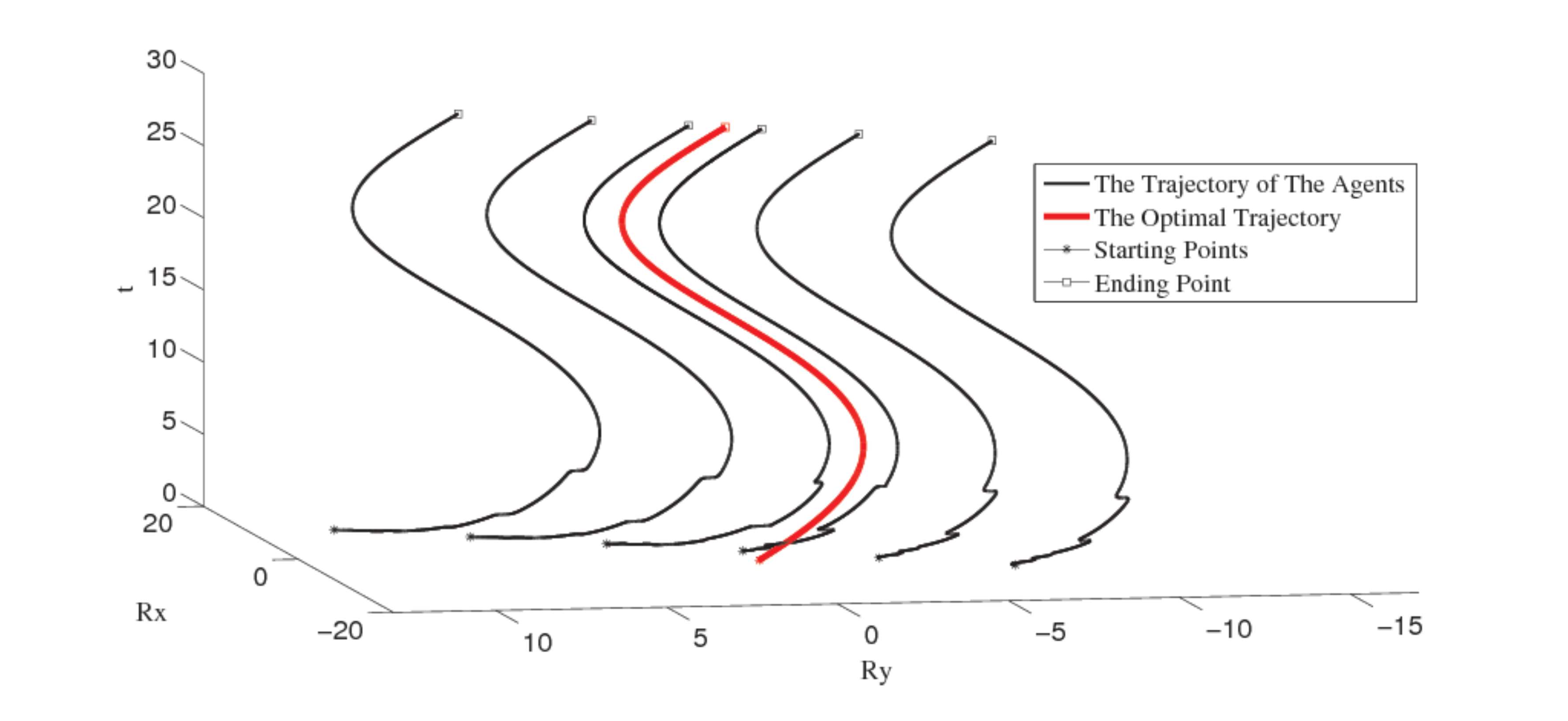}}}
\vspace{-.6cm} \caption{Trajectories of all agents along with optimal trajectory using algorithm \eqref{usingleswarm} for local cost function \eqref{costfunc1}} \vspace{-0.1cm}
\label{sgldisfig}
\end{center}
\end{figure}

In our next illustration, the swarm control algorithm \eqref{uswarm} is employed for double-integrator dynamic system \eqref{double} to minimize the team cost function where the local cost functions are defined as
\begin{equation} \label{costfunc3} \small
f_i(x_i(t),t)= (r_{x_i}(t) +2i \frac{\text{sin}(0.5t)}{t+1})^2+ (r_{y_i}(t) +i\text{sin}(0.1t))^2,
\end{equation} \normalsize
In this case, the parameters of control algorithm \eqref{uswarm} are chosen as $\alpha=10$ and $\beta=20$. Fig. \ref{dblswarmfig} shows that the center of the agents' positions tracks the optimal trajectory while the agents remain connected and avoid collisions.

\begin{figure}[t] 
\begin{center} \hspace*{-.7cm}
\vspace{0.1cm} {\scalebox{0.30}{\includegraphics*{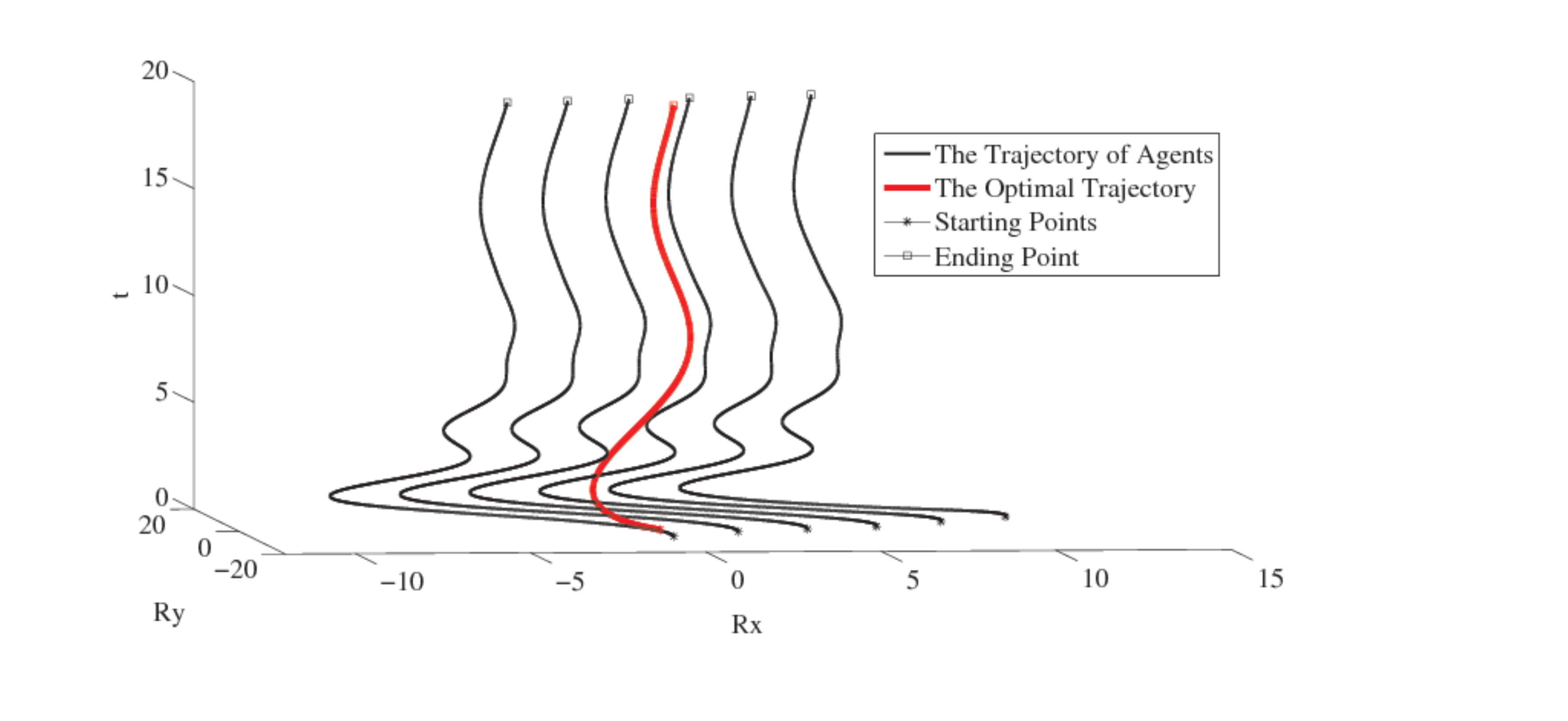}}}
\vspace{-0.8cm} \caption{Trajectories of all agents using algorithm \eqref{uswarm} for local cost function \eqref{costfunc3}} \vspace{-0.3cm}
\label{dblswarmfig}\end{center}
\end{figure}

\section{Conclusions} \label{sec:conclusions}
In this paper, a distributed convex optimization problem with swarm tracking behavior was studied for continuous-time multi-agent systems. The agents' task is to drive their center to minimize the sum of local time-varying cost functions through local interaction, while maintaining connectivity and avoiding inter-agent collision. Each local cost function is known to only an individual agent. Two cases were considered, single-integrator dynamics and double-integrator dynamics. In each case, a distributed algorithm was proposed where each agent relies only on its own position and the relative positions (and velocities in double-integrator case) between itself and its neighbors. Using these algorithms, it was proved that for both cases the center of agents tracks the optimal trajectory while the connectivity of the agents was maintained and agents avoided inter-agent collision.

\bibliographystyle{IEEEtran}
\bibliography{IEEEabrv,IEEE_TSG_00018_2010}

\end{document}